\documentclass[12pt]{amsart}
\usepackage{amssymb, latexsym, amsthm} %, makeidx}

\newtheorem{thm}{Theorem}[section] %the resolution could also be [subsection]

\newtheorem{cor}[thm]{Corollary}
\newtheorem{defn}[thm]{Definition}

\newtheorem{prop}[thm]{Proposition}

\newtheorem{rem}[thm]{Remark}

\DeclareMathOperator{\ed}{ed}
\DeclareMathOperator{\Cor}{Cor}

\DeclareMathOperator{\ind}{ind}
\DeclareMathOperator{\SB}{SB}
\DeclareMathOperator{\M}{M}

\DeclareMathOperator{\ch}{char}
\DeclareMathOperator{\Brdim}{Br.dim}

\DeclareMathOperator{\Br}{Br}
\DeclareMathOperator{\Tr}{Tr}
\DeclareMathOperator{\trdeg}{trdeg}
\DeclareMathOperator{\Nrd}{Nrd}

\DeclareMathOperator{\N}{N}

\def\ra{{\rightarrow}}

\newcommand\operA[2]{{\if!#2!\operatorname{#1}\else{\operatorname{#1}_{#2}^{\phantom{I}}}\fi}} % To be used within Bdefs. Usage: $\operA{N}{K/F}$ produces $N_{K/F}$; $\operA{N}{}$ produces $N$.

\def\tr{{\operatorname{Tr}}}

\newcommand\mul[1]{{#1^{\times}}} % The multiplicative group

\newcommand{\Trace}[1][]{\if!#1!\operatorname{Tr}\else{\operatorname{Tr}_{#1}^{\phantom{I}}}\fi} % Usage: $\Tr[K/F](a)$.

\newcommand\book[4]{{{#1},\ {{#2}}{\if!#3!\relax\else{,\ {#3}}\fi}{\if!#4!\relax\else{,\ {#4}}\fi}.}} % Format: \book{author}{title}{info}{year}
\newcommand\st{{\,|\:}}
\newcommand\co{{\,:\,}}
\newcommand\len{{\operatorname{len}}}
\newcommand\isom{{\,\cong\,}}

\newif\ifXY % turns XY version on/off
\XYtrue     % Turn it on
\ifXY

\input xy
\input xyidioms.tex
\usepackage{xy}
\xyoption{all} %
\fi % For \ifXY

\begin{document}
\title{Symbol length in the Brauer group of a field}
\author{Eliyahu Matzri}

\thanks{The author thanks Daniel Krashen, Andrei Rapinchuk, Louis Rowen, David Saltman and Uzi Vishne for all their help, time and support.}
\thanks{This work was partially supported by the BSF, grant number 2010/149.}

\begin{abstract}
We bound the symbol length of elements in the Brauer group of a field $K$ containing a $C_m$ field (for example any field containing an algebraically closed field or a finite field), and solve the local exponent-index problem for a $C_m$ field $F$. In particular, for a $C_m$ field $F$, we show that every $F$ central simple algebra of exponent $p^t$ is similar to the tensor product of at most $\len(p^t,F)\leq t(p^{m-1}-1)$ symbol algebras of degree $p^t$. We then use this bound on the symbol length to show that the index of such algebras is bounded by $(p^t)^{(p^{m-1}-1)}$, which in turn gives a bound for any algebra of exponent $n$ via the primary decomposition. Finally for a field $K$ containing a $C_m$ field $F$, we show that every $F$ central simple algebra of exponent $p^t$ and degree $p^s$ is similar to the tensor product of at most $\len(p^t,p^s,K)\leq \len(p^t,L)$ symbol algebras of degree $p^t$, where $L$ is a $C_{m+\ed_L(A)+p^{s-t}-1}$ field.

\end{abstract}

\maketitle

%-----------------------------------------------------------------
% \setcounter{tocdepth}{3}
% \tableofcontents

%
\section{Introduction}
%
%A field is called $C_m$ if it has the property that every homogeneous equation of degree $d$ in more then $d^m$ variables has a non-trivial solution. It was conjectured that these fields should have some universal bounds on the symbol length and exponent-index relations but not much was known except for the case of a $C_2$ field where Artin in \cite{Art82} showed that the tensor product of any two algebras of degree two or three is similar to one algebra of the same degree respectively and deduced that the Brauer dimension at these primes is one.
%The main theorems we prove in this work are: \\

We are interested in the following two problems:\\
\textbf{The symbol length problem}:
Let $F$ be a field and $A$ a $F$ central simple algebra of exponent~$n$.
Assuming $F$ contains a primitive $n$-th root of one $\rho_n$, the Merkurjev-Suslin theorem tells us that any such $A$ is brauer equivalent to the tensor product of symbol algebras of degree~$n$.
The minimal number of symbol algebras needed is called the symbol length of $A$ denoted $\len(n,A)$.
The symbol length problem asks if there is a finite bound $\len(n,F)$, such that for any $A\in \Br(F)$ of exponent $n$ one has $\len(n,A)\leq \len(n,F)$. One can filter the $n$-th torsion of the brauer group by degree and define $\len(n,m,F)$ or $\len(n,m)$ when it is independent of the field $F$, as the minimal number of symbols needed to express every $A\in \Br(F)$ of exponent $n$ and degree~$m$. Notice that the existence of a generic division algebra of exponent $n$ and degree $m$ implies $\len(n,m)$ is always finite. Finding an explicit bound for $\len(n,m)$ is also referred to as the symbol length problem.\\
\textbf{The exponent-index problem}:
It is well known that for any \linebreak $A\in \Br(F)$ the exponent of $A$ divides the index of $A$ and the two numbers have the same prime factors; in particular the exponent is bounded by the index.
The exponent-index problem asks if one can bound the index in terms of the exponent.
To be more precise, for a prime $p$ define the Brauer dimension at $p$, denoted $\Brdim_p(F)$, to be the smallest integer $d$ such that for all $n\in \mathbb{N}$ and $A\in \Br_{p^n}(F)$, $\ind(A)$ divides $\exp(A)^d$, and $\infty$ if no such number exists. Then define the Brauer dimension of $F$ to be $\Brdim(F)=\sup \{ \Brdim_p(F) \}$.
The global exponent-index problem asks if $\Brdim(F)$ is finite and the local exponent-index problem asks if $\Brdim_p(F)$ is finite.

The answer to these problems is negative for arbitrary fields. To see this consider the field $F=\mathbb{Q}[\rho_p](x_1,y_1,...,x_i,y_i,...)$ and define\\ $A_n=\otimes^{n}_{i=1}(x_i,y_i)_{F,p}$. Then it is known that $A_n$ is a division algebra (see for example \cite[Corollary 1.2]{dJ04}), that is $\ind(A_n)=p^n$, and $\exp(A_n)=p$.\\ In particular $\len(p,A_n)=n$ and $\Brdim_p(F)\geq n$ for all $n\in \mathbb{N}$, implying $\len(p,F)=\infty$ and $\Brdim(F)=\Brdim_p(F)=\infty$.

 It seems that a positive answer to these problems is strongly related to the arithmetic of the base field $F$. This is supported by the following results:
\begin{enumerate}
\item For $F$ a local or global field, $\Brdim(F)=1$ by the Albert-Brauer-Hasse-Noether theorem \cite{A0} and \cite{BHN}.
\item For $F$ a $C_2$ field, M. Artin conjectured \cite{Art82} that $\Brdim(F)~=~1$. 
	He proved that $\Brdim_2(F)= \Brdim_3(F)= 1$ for such fields. 
\item For $F$ a finitely generated field of transcendence degree $2$ over an algebraically closed field, $\Brdim(F)=1$ by \cite{dJ04} and \cite[Theorem 4.2.2.3]{Lie08}.
\item For $F$ finitely generated and of transcendence degree 1 over an $\ell$-adic field,
$\Brdim_p(F)= 2$ for every prime $p\neq \ell$ by \cite{Sal97a}.

\end{enumerate}

Motivated by M. Artin's results (\cite{Art82}) we focus our attention on a class of fields called $C_m$ fields.
A field $F$ is called $C_m$ if it has the property that every homogeneous equation $f(x_1,...,x_n)=0$, of degree~$d$ has a non-trivial solution when $n>d^m$.

We solve both the symbol length and the local exponent-index problems for these fields by giving explicit bounds on $\len(n,F)$ and $\Brdim_p(F)$.
In particular we prove the following theorems,\\
\textbf{Theorem 4.4} 
Let $F$ be a $C_m$ field containing all $n$-th roots of unity and $A\in \Br_n(F)$ be of exponent $n=p^t$. Then:
$A\sim \otimes_{i=1}^t C_i$ where $C_i=\otimes_{j=1}^{p^{m-1}-1} (\alpha_{i,j},\beta_{i,j})_{p^{i}}$. In particular  $\len(p^t,F)\leq t(p^{m-1}-1)$. \\
\textbf{Theorem 5.3} %\ \\
Let $F$ be a field containing a $C_m$ field $L$ and all\linebreak $p^t$-th roots of unity, and $A$ be a $F$-csa of exponent $p^t$ and degree $p^s$. Then the symbol length of $A$ is bounded by $\len(p^t,K)$ where $K$ is a $C_{m+\ed_L(A)+p^{s-t}-1}$ field. \\
\textbf{Theorem 6.3} % \ \\
If $F$ is a $C_m$ field, then $\Brdim_p(F)\leq p^{m-1}-1$. \\
\textbf{Theorem 8.2} %\ \\
Let $F$ be a $C_m$ field and let $\alpha \in K^M_2(F)/nK^M_2(F)$ where $n=p^t$, then $\alpha$ can be written as the sum of at most $t(p^{m-1}-1)$ symbols. \\

The approach we take is to first bound the symbol length and then use this bound to get a bound for $\Brdim_p(F)$.
To bound the symbol length we start with $A\in \Br_{p^n}(F)$, use the Merkurjev-Suslin theorem to assume it is a product of symbol algebras, and then show how to shorten the number of symbol algebras down to a fixed number.
The key idea is to consider $A_k=\otimes_{i=1}^k (\alpha_i, \beta_i)$ for $k\in\mathbb{N}$ and produce ``large" vector spaces $V_k\leq A_k$ called $n$-Kummer spaces with the property that for every $v\in V_k$ one has $v^n\in F$. These spaces have ``low" degree norm forms, $\N_i:V_i\rightarrow F$ defined by $\N_i(v)=v^n$. Thus when $k$ is big enough the $C_m$ property ensures the existence of a non trivial solution for $\N_k(v)=0$ from which we deduce how to shorten the number of symbols.

%\newpage
The paper is organized as follows:
We start with a background section where we give the main definitions needed for this work and some known results in the subject.
In section~$3$ we use $n$-Kummer spaces and their norm forms to get bounds on the symbol length for arbitrary exponent $n$.
In section~$4$ we use known results about primary decomposition in the Brauer group and a divisibility property of symbol algebras to improve the bounds obtained in section~$3$. 
In section~$5$ we generalize the discussion to fields containing a $C_m$ for some $m$. 
Section~$6$ is devoted to the exponent-index problem, where we use the results in section~$4$ to solve the local exponent-index problem. 
Section~$7$ is devoted to the characteristic $p>0$ case.
Finally in the section~$8$ we show that our results can be formulated in the context of the second Milnor $K$-group where the presence of roots of unity is not required.

%\newpage

%
%\newpage
%
\section{Background}
%Throughout this work $F$ is a field.
%\subsection{$C_m$-fields}
%\begin{defn}(\cite{Serre} page 87)
%A field, $F$, is called a $C_m$ if every homogenouse equation, $f(x_1,...x_n)=0$, of degree $d$ has a non-trivial solution if $n>d^m$.
%\end{defn}
%$C_m$ fields are not as rare as it might seem from the definition, here are some known examples:
%\begin{enumerate}
%\item Every algebraically closed field is $C_0$.
%\item Every finite field is $C_1$. 
%\item If $F$ is $C_m$ and $F\subset K$ is of transcendence degree $n$ over $F$, then $K$ is $C_{m+n}$, \cite{Lang} completed by \cite{Nagata}.
%\item The above implies that if $V$ is a variety of dimension $n$ over an algebraically closed field $F$, then the function field, $F(V)$, is $C_n$.
%\end{enumerate}

\subsection{The Brauer group}
Let $F$ be a field. A $F$-central simple algebra, denoted $F$-csa, is an $F$-algebra, simple as a ring with center $F$.
The Brauer group of $F$ is defined as $$\{ \hbox{isomorphism classes of finite dimensional $F$-csa} \}/\sim$$
where for two $F$-csa $A$ and $B$, $$A\sim B\Leftrightarrow \exists n,m\in \mathbb{N} : \M_n(A)\cong \M_m(B)$$ 
It is well-known that $\Br(F)$ is a torsion group. We write $\exp(A)$ for the order of $A$ in $\Br(F)$ and $\Br_n(F)$ for the $n$-torsion subgroup of $\Br(F)$.
By the Wedderburn-Artin theorem every $F$-csa $A$ is isomorphic to $\M_n(D)$ for unique $n\in \mathbb{N}$ and $F$-central division algebra $D$ called the underlying division algebra of $A$.
One defines the degree and index of~$A$ as $\deg(A)=\sqrt{\dim_F(A)}$ and $\ind(A)=\deg(D)$ respectively.  
For a more detailed study of the Brauer group  we refer the reader to \cite{RB}, \cite{Jac} or \cite{GS} .

%
%Denote $$\CSA(F)=\{ \hbox{isomorphism classes of finite dimensional $F$-csa} \}$$
%By the Wedderburn-Artin theorem every $A\in \CSA(F)$ is isomorphic to $\M_n(D)$ for unique $n\in \mathbb{N}$ and central division algebra $D$-called the underlying division algebra of $A$.
%Since there are no non-trivial finite dimensional division algebras with center, a separably closed field, passing to the separable closure of $F$, we see $A\cong \M_d(F^{sep})$ implying the dimension of $A$ over $F$ is a perfect square.
%Thus for $A\in \CSA(F)$ one defines the degree and index of $A$ with underlying division algebra $D$ as $\deg(A)=\sqrt{\dim_F(A)}$ and $\ind(A)=\deg(D)$ respectively.  
%One defines the Brauer group of $F$ as the quotient
%$$\Br(F)=\CSA(F)/\sim $$ where for $A,B\in\CSA(F), A\sim B\Leftrightarrow \exists n,m\in \mathbb{N} : \M_n(A)\cong \M_m(B)$.
%We write $[A]$ for the class of $A$ in $\Br(F)$ but we sometimes abuse notations and write $A$ for $[A]$.
%$\Br(F)$ is an abelian group with respect to the tensor product, where $F$ is the trivial element and $[A]^{-1}=[A]^{op}$ , the opposite algebra of $A$. We write $\exp(A)$ for the order of $A$ in $\Br(F)$, which is finite since $\Br(F)$ is torsion, and $\Br_n(F)$ for the n-torsion subgroup of $\Br(F)$.
%Notice that by the uniqueness of the underlying division algebra if $[A]=[B]$ then $\ind(A)=\ind(B)$ thus it makes sense to define $\ind([A])=\ind(A)$. For more on the subject we refer the reader to \cite{RB}.
An important example of $F$-csa are symbol algebras which we now define.
Let $F$ be a field containing a primitive $n$-th root of $1$ denoted~$\rho _n$ and $a,b\in F^{\times}$.
Define the symbol algebra $$(a,b)_{n,F}=F[x,y|x^n=a, y^n=b, yx=\rho_nxy].$$ Then $(a,b)_{n,F}$ is a $F$-csa of degree $n$ and exponent dividing $n$.\\
A standard pair of generators for $(a,b)_{n,F}$ is a pair $u,v\in (a,b)_{n,F}$ satisfying $u^n\in \mul{F}, \ v^n\in \mul{f}$ and $uv=\rho_n vu$, and for any such pair one has $(a,b)_{n,F}\cong (v^n,u^n)_{n,F}$.

%Notice that $(a,b)_{n,F}$ contains a Galois extension of $F$, $F[x]$ (which might be etale) with 
%$\Gal(F[x]/F)=<\sigma: x\rightarrow \rho_nx>\cong C_n$, and every $F$-csa of degree $n$ containing such an extension of $F$ has a representation as a symbol algebra.
The most famous example is the well known quaternion algebra which has the presentation $(-1,-1)_{2,\mathbb{R}}$.
The presentation of a $F$-csa as a symbol algebra (or the tensor product of several symbol algebras) is not unique, and starting from a given presentation one can produce many others.
The following proposition tells us that we may assume one of the slots represents a field, that is,
 
\begin{prop}\label{P}
Assume $A=(a,b)_{n}$ does not split. Then we can modify the presentation of $A$ such that $F[x]$ is a field where $x^n=a$. 

\end{prop}

\begin{proof}
%Write $n=\prod_{i=1}^t p_i^{n_i}$ where $p_i$ are distinct primes. It is not hard to see that $K=\prod_{i=1}^t F[[x_i|x_i^{n_i}=c^{\frac{n}{n_i}]}$ so we may assume $n=p^n$
It is enough to consider the case where $n=p^m$. Now $K = F[x]$ is a field if and only if $x^n-a$ is an irreducible polynomial if and only if $a \not \in {F ^{\times}}^p$. 
Let $s=\max \{ s| a \in {F ^{\times}}^{p^s}  \}$, which is finite because $a \not \in {F ^{\times}}^{p^m}$, as $A$ does not split. Write $a=c^{p^s}$. Then by the definition of $s$ we know $c \not \in {F ^{\times}}^p$. Now write $A=(a,b)_{n}=(c^{p^s},b)_{n}=(c,b^{p^s})_{n,F}$ and now $K=F[x|x^n=c]$ is a field.
\end{proof}
%\newpage

We give some well known relations which will be used to prove the main theorems of this work.
\begin{prop}\label{A}
Let $A_i=(a_i,b_i)_n$ with standard generators $x_i,y_i$,  $i=1,2$, and let $\N_{F[x_i]/F}$ denote the regular field norm. 
\begin{enumerate}
\item \label{Ab} For every $k_1\in \mul{F[x_1]}$, $A_1\cong(a_1,\N_{F[x_1]/F}(k_1)b_1)$. 
\item \label{A0} If $a_1+b_1\neq 0$ then,  $A_1\cong (a_1+b_1,-a_1^{-1}b_1)_n$.
\item \label{A1} $A_1\otimes A_2\cong (a_1,b_1b_2^{-1})_n\otimes (a_1a_2,b_2)_n$.
\item \label{A2} For $k_2\in \mul{F[x_2]}$, if $t = a_1a_2+\N_{F[x_2]/F}(k_2) b_2\neq 0$ then \\ $A_1\otimes A_2\cong (a_1,*)_n \otimes (t,*)_n$. %, where the $*$'s will be given in the proof.
\end{enumerate}
\end{prop}

\begin{proof}  \ \\
(1)+(2) are standard relations which can be found in \cite{RB}, \cite{Jac} or \cite{GS}.\\
(3) Consider the commuting pairs $$u_1=x_1,v_1=y_2^{-1}y_1  \ \hbox{and} \ u_2=x_1x_2,v_2=y_2,$$ noting that $u_1^n=a_1$, $v_1^n=b_2^{-1}b_1$, $u_2^n=a_1a_2$, $v_2^n= b_2$.\\
(4) Combining (\ref{Ab}), (\ref{A0}) and (\ref{A1}) we have,
\begin{align*}
A_1\otimes A_2 &= (a_1,b_1)_n\otimes (a_2,b_2)_n\\
& \stackrel{(\ref{Ab})}{\cong} (a_1,b_1)_n\otimes (a_2,\N_{F[x_2]/F}(k_2) b_2)_n \\  
& \stackrel{(\ref{A1})}{\cong} (a_1,b_1(\N_{F[x_2]/F}(k_2) b_2)^{-1})_n\otimes (a_1a_2,\N_{F[x_2]/F}(k_2) b_2)_n \\
& \stackrel{(\ref{A0})}{\cong} (a_1,b_1(\N_{F[x_2]/F}(k_2) b_2)^{-1})_n\otimes (a_1a_2+\N_{F[x_2]/F}(k_2) b_2,-(a_1a_2)^{-1}\N_{F[x_2]/F}(k_2) b_2)_n\\
&=(a_1,*)_n \otimes (t,*)_n
\end{align*}
as claimed.
%\item \label{A1} $A_1\otimes A_2\cong (a_1,b_1b_2^{-1})_n\otimes (a_1a_2,b_2)_n$
%\item \label{A2} If $t = a_1a_2+\alpha b_2\neq 0$ then $A_1\otimes A_2\cong (a_1,*)_n \otimes (t,*)_n$, where the *'s will be given in the proof.
%%\end{enumerate}
%\\ $A_1\otimes A_2= (a_1,b_1)_n\otimes (a_2,b_2)_n \stackrel{(\ref{Ab})}{\cong} (a_1,b_1)_n\otimes (a_2,\N_{F[x_2]/F}(k_2) b_2)_n \stackrel{(\ref{A1})}{\cong} \\(a_1,b_1(\N_{F[x_2]/F}(k_2) b_2)^{-1})_n\otimes (a_1a_2,\N_{F[x_2]/F}(k_2) b_2)_n\stackrel{(\ref{A0})}{\cong} \\(a_1,b_1(\N_{F[x_2]/F}(k_2) b_2)^{-1})_n\otimes (a_1a_2+\N_{F[x_2]/F}(k_2) b_2,-(a_1a_2)^{-1}\N_{F[x_2]/F}(k_2) b_2)_n= \\(a_1,*)_n \otimes (t,*)_n$ 

\end{proof}

\subsection{Severi-Brauer Varieties}
Let $A$ be a $F$-csa of degree $n$.
The Severi-Brauer variety associated to $A$ denoted $\SB(A)$, is the variety of all minimal left ideals of $A$.
The dimension of $\SB(A)$ is $n-1$.
This variety contains the splitting information for $A$ as seen from the following theorem:
\begin{thm}(\cite[Theorem 13.7]{SLN})
Let $A$ be as above. The following are equivalent:
\begin{enumerate}
\item $A\isom \M_n(F)$, i.e. $A$ is split.
\item $\SB(A)\isom \mathbb{P}^{n-1}(F)$.
\item $\SB(A)$ has a rational point.
\end{enumerate}
\end{thm}

One more important property of this variety is that its function field serves as a generic splitting field for $A$ in the following sense:
\begin{thm} (\cite[13]{SLN})
The following are equivalent:
\begin{enumerate}
\item $K$ is a splitting field for $A$.
\item There is a place $\nu:F(\SB(A))\rightarrow K$.
\end{enumerate}
\end{thm}
For more on this important variety we refer the reader to \cite{SLN} or \cite{GS}.

\subsection{$C_m$ fields}

%\begin{defn}(\cite[p. 87]{Serre})
%A field, $F$, is called a $C_m$ if every homogeneous equation, $f(x_1,...,x_n)=0$, of degree $d$ has a non-trivial solution if $n>d^m$.
%\end{defn}
Even though the definition of a $C_m$ field seems quite restrictive there are many interesting fields which are $C_m$. Here are some known examples:
\begin{enumerate}
\item Every algebraically closed field is $C_0$.
\item Every finite field is $C_1$. 
\item If $F$ is $C_m$ and $F\subset K$ is of transcendence degree $n$ over $F$, then $K$ is $C_{m+n}$, by \cite{Lang} completed by \cite{Nagata}.
\item The above implies that if $V$ is a variety of dimension $n$ over an algebraically closed field $F$, then the function field, $F(V)$, is~$C_n$.
\end{enumerate}

%\newpage

\subsection{Known results}  

%Results on the exponent-index problem:\\
%
%\begin{enumerate}
	%\item For $F$ a local or global field $\Brdim(F)=1$ by the Albert-Brauer-Hasse-Noether theorem \cite{A0} and \cite{BHN}.
	%\item For $F$ a finitely generated of transcendence degree $2$ over an algebraically closed field, $\Brdim(F)=1$ by \cite{dJ04} and \cite[Theorem 4.2.2.3]{Lie08}.
	%\item M. Artin conjectured in \cite{Art82} that $\Brdim(F)=1$ for
%every $C_2$ field $F$. He proved that $\Brdim_2(F)= \Brdim_3(F)= 1$ for such
%fields. 
	%\item For $F$ finitely generated and of transcendence degree 1 over an $\ell$-adic field
%$\Brdim_p(F)= 2$ for every prime $p\neq \ell$ by \cite{Sal97a}.
%\item If $F$ is a complete discretely valued field with residue field $k$ such that
%$\Brdim_p(k) \leq d$ for all primes $p\neq \ch(k)$, then $\Brdim_p(F)\leq d+1$ for all
%$p \neq \ch(k)$ by \cite[Theorem 5.5]{HHK09}.
	%\item If $F$ has characteristic p and is finitely generated of transcendence degree $r$ over a perfect
%field $k$, then $\Brdim_p(F)\leq r$, by methods of Albert (\cite[page 7]{Ketura}).
%
%\end{enumerate}

Results on symbol length:
\begin{enumerate} 
	\item Every algebra of degree $2$ is isomorphic to a quaternion algebra. That is, $\len(2,2)=1$.
	\item Every algebra of degree $3$ is cyclic and thus, when $\rho_3\in F$ it is isomorphic to a symbol algebra. That is, $\len(3,3)=1$ (Wedderburn \cite{Wed21}).
	\item Every algebra of degree $4$ of exponent $2$ over a field of characteristic different from $2$ is isomorphic to a product of two quaternion algebras. That is, $\len(4,2)=2$ (Albert \cite{Al}).
	\item Every algebra of degree $8$ and exponent $2$ is similar to the product of four quaternion algebras. That is, $\len(8,2)=4$ (Tignol~\cite{T2}).
	\item Every algebra of degree $9$ and exponent $9$ over a field of characteristic different than $3$ containing $\rho_9$ is similar to the product of $35840$ symbol algebras of degree $9$ and if it is of exponent $3$ it is similar to the product of $277760$ symbol algebras of degree $3$. That is, $\len(9,9)\leq 35840$ and $\len(9,3)\leq 277760$ (Matzri~\cite{MAT}).
\item Every algebra of prime degree $p$ over a field of characteristic different from $p$ containing $\rho_p$ is similar to the tensor product  of $\frac{(p-1)!}{2}$ symbol algebras. That is, $\len(p,p)\leq \frac{(p-1)!}{2}$ (Rosset-Tate \cite{RT},\cite[7.4.11]{GS} and Rowen-Saltman \cite{RSD}).	
	
	%\item Every abelian crossed product with respect to $(\mathbb{Z}_2)^4$  of exponent $2$ is similar to the product of $18$ quaternion algebras (Sivatski \cite{SV}).
	\item Every $p$-algebra of index $p^n$ and exponent $p^m$ is similar to the product of $p^n-1$ cyclic algebras of degree $p^m$. That is,\\ $\len(p^m,p^n)=p^n-1$ (Florence \cite{MF}).
	\item If $F$ is the function field of an $l$-adic curve containing a primitive $p$-th root of one and $p$ is a prime different than $l$, then every degree $p$ algebra is cyclic. That is, $\len(p,p,F)=1$ (Saltman~\cite{Sal97a}). 
	\item If $F$ is a local or global field containing a primitive $n$-th root of~$1$, every algebra of exponent $n$ is a symbol. That is $\len(n,F)=1$ (Albert-Brauer-Hasse-Noether  \cite{A0} and \cite{BHN}).
	\item If $F$ is a $C_2$ field containing a necessary primitive root of~$1$ then, $\len(2,F)=\len(3,F)=1$ (Artin \cite{Art82}).
	\item If $F$ is the function field of an $l$-adic curve and $(n,l)=1$, then every algebra of exponent $n$ is the product of two cyclic algebras. Thus if $F$ contains a primitive $n$-th root of~$1$, $\len(n,F)=2$ (Brussel, Mckinnie and Tengan \cite{BMT}).
\end{enumerate} \ \\

Results on the exponent-index problem:\\

\begin{enumerate}
	\item For $F$ a local or global field, $\Brdim(F)=1$ (Albert-Brauer-Hasse-Noether \cite{A0} and \cite{BHN}).
	\item For $F$ a finitely generated of transcendence degree $2$ over an algebraically closed field, $\Brdim(F)=1$ (de-Jong \cite{dJ04} and Lieblich \cite[Theorem 4.2.2.3]{Lie08}).
	\item M. Artin conjectured in \cite{Art82} that $\Brdim(F)=1$ for
every $C_2$~field~$F$. He proved that $\Brdim_2(F)= \Brdim_3(F)= 1$ for such
fields. 
	\item For $F$ finitely generated and of transcendence degree 1 over an $\ell$-adic field,
$\Brdim_p(F)= 2$ for every prime $p\neq \ell$ (Saltman~\cite{Sal97a}).
\item If $F$ is a complete discretely valued field with residue field $k$ such that
$\Brdim_p(k) \leq d$ for all primes $p\neq \ch(k)$, then $\Brdim_p(F)\leq d+1$ for all
$p \neq \ch(k)$ (Harbater, Hartmann and Krashen~\cite[Theorem 5.5]{HHK09}).
	\item If $F$ has characteristic p and is finitely generated of transcendence degree $r$ over a perfect
field $k$, then $\Brdim_p(F)\leq r$, by methods of Albert (\cite[page 7]{Ketura}).

\end{enumerate}

%
%For every $A$ in $\Br(F)$ it is well known that $\ind(A)|\exp(A)$ and the two numbers have the same prime factors.
%The other implication is called the index-exponent problem.
%More explicitly, we adopt the notation of cite{Ketura} and define the Brauer dimension of a field $F$ as:
%$$\Brdim(F)=\min \{n\in \mathbb{N}| \ind(A)|\exp(A)^n, \forall A \in \Br(F)\}$$
%One can localize this notion and define the Brauer dimension at a prime $p$ as 
%$$\Brdim_p(F)=\min \{n\in \mathbb{N}| \ind(A)|\exp(A)^n, \forall A \in \Br(F) \hbox{with} \exp(A)=p^k\}$$
%The index-exponent problem is then: \\
%Is $\Brdim(F)<\infty?$ or localized, Is $\Brdim_p(F)<\infty?$\\
%
%Notice that for a general field, $F$, the answer is NO. To see this consider the field $F=\mathbb{Q}[\rho_p](x_1,y_1,...,x_i,y_i,...)$and define $A_n=\otimes^{n}_{i=1}(x_i,y_i)_{F,p}$, then it is known that $A_n$ is a division algebra, that is $\ind(A_n)=p^n$, and $\exp(A_n)=p$. Thus $\Brdim(F)\geq \Brdim_p(F)\geq n$ for all $n\in \mathbb{N}$ implying $\Brdim(F)=\Brdim_p(F)=\infty$.
%
%However over specific fields things might be better, here are some of the results known:
%\begin{enumerate}
%\item \cite{ABHN} 
%\item \cite{Saltman} p-adic 
%\item \cite{dejong} surfaces over algebraicly closed fields.
%\item \cite{lieblich}
%\item \cite{Artin}
%
%
%
%\end{enumerate}
%\newpage

%
%
%It is conjectured that for a $C_m$ field, $F$, $\Brdim(F)=m-1$ this is supported by the above results of $2,3$ but not much more is known.
%We will prove  $\Brdim_p(F)\leq p^{m-1}-1$ thus showing that locally the Brauer dimension is finite.

%\newpage

\section{Using $n$-Kummer spaces to bound the symbol length}

%Let $n$ be a natual number and let $F$ be a field containing a primitive $n$-th root of unity $\rho$.  By the Merkurjev-Suslin theorem, every central simple algebra $A$ of exponent $n$ over $F$ is Brauer equivalent to a tensor product of symbol algebras of degree $n$. The minimal number of such symbols is the {\bf{symbol length}} of $A$, denoted by $\len(A)$. We show that the symbol length of any algebra over a $C_m$ field is bounded in terms of $n$ and $m$ (independenly of the index of the algebra), we denote this number as $\len(m,n)$.

As above, for $a,b \in \mul{F}$, we denote by $(a,b)_n$ the symbol algebra $F[x,y \st x^n=a, \,y^n=b\,yx=\rho xy]$.

We say that an $F$-subspace $V$ of a central simple algebra $A$ is $n$-Kummer if $v^n \in F$ for every $0\neq v \in V$. An $n$-Kummer space $V$ is endowed with the exponentiation map $$N_V \co V \ra F, \ \hbox{defined by} \ \N_V(v) = v^n$$ which is a homogeneous form of degree $n$.  
\begin{rem} \  \ \ \ \ \ \ \ \ \ \ \ \ \ \ \ \ \ \ \ \ \ \ \ \ \ \ \ \ \ \ \ \ \ \ \ \ \ \ \ \ \ \ \ \ \ \ \ \ \ \ \ \ \ \ \ \ \ \ \ \ \ \ \ \ \ \ \ \ \ \
\begin{enumerate}
\item If $\deg(A)=nm$, then $\N_V(v)^m=\Nrd_A(v)$.
\item An element $d\in A$ of degree $n$ has charachteristic polynomial $\lambda ^n-\alpha$ (thus satisfies $d^n\in F$) if and only if $tr(d^m)=0$ for all $1\leq m \leq n-1$, where $\tr(d)$ is the usual field trace (this is clear by Newton's inversion formulas).
\item If $x,y$ satisfy that $x^n,y^n\in F$ and $yx=\rho_n xy$ then, $(x+y)^n=x^n+y^n\in F$. Indeed one can check that $\Tr((x+y)^m)=0$ for all $1\leq m \leq n-1$, (which is clear as $\tr(x^iy^j)=0$ for $(i,j)\neq (0,0)$ mod $n$) thus $(x+y)^n\in F$. On the other hand when explicitly computing $(x+y)^n$ one gets $x^n+y^n+ M$ where $M$ is a sum of monomials of the form $f_{i,j} x^iy^j$ for $1\leq i,j\leq n-1$ which are linearly independent and not in $F$ thus we conclude that 
$f_{i,j}=0$ for all $1\leq i,j\leq n-1$ and $(x+y)^n=x^n+y^n\in F$. 
\end{enumerate}
\end{rem} \ \\ \\ \\ 
\textbf{Examples}\ \ \ \ \ \ \ \ \ \ \ \ \ \ \ \ \ \ \ \ \ \ \ \ \ \ \ \ \ \ \ \ \ \ \\ \ \ \ \ \ \ \ \ \ \ \ \ \ \ \ \ \ \ \ \ \ \ \ \ \ \ \ \ \ \ \ \ \ \
Consider $A=(a,b)_n$ with standard generators $x,y$.
\begin{enumerate}
\item $V_1=Fx$ and $V_2=Fy$ are by definition one dimensional $n$-Kummer spaces with norms $\N_1(fx)=f^na$ and $\N_2(fy)=f^nb$ respectively.
\item $V=Fx+Fy$ is a two dimensional $n$-Kummer space with norm \linebreak $\N_V(fx+gy)=f^na+g^nb$.
\item $V=F[x]y$ is an $n$-dimensional $n$-Kummer space with norm \linebreak $\N_V(ky)=\N_{F[x]/F}(k)b$ for $k\in F[x]$.
\item \label{SS1} $V=Fx+F[x]y$ is an $n+1$-dimensional $n$-Kummer space with norm  $\N_V(fx+ky)=f^na+\N_{F[x]/F}(k)b$ for $k\in F[x]$.
\end{enumerate}
Our objective is to find high dimensional $n$-Kummer spaces so that the $C_m$-property will ensure a non-trivial solution to the norm form. However, Example (\ref{SS1}) is maximal with respect to inclusion (\cite{ACT}) and we actually conjecture that if $A$ is a division algebra the maximal dimension of such a space is $n+1$. Thus we consider tensor products of symbol algebras.

Let $A$ be the tensor product $\otimes_{i=1} ^{t} (a_i,b_i)_n$, with the standard pairs of generators $x_i,y_i$ for the symbol algebras. Let $V_0 = F$ and for \linebreak $j = 1,\dots,t$ let $V_j\subset \otimes_{i=1} ^{j} (a_i,b_i)_n$ be defined by $$V_{j}=V_{j-1}x_{j}+F[x_{j}]y_{j}$$ (so in particular $V_1=Fx_1+F[x_1]y_1$). These are called standard \linebreak $n$-Kummer spaces.\\
Every $v_k\in V_k$ defines two vectors $\widetilde{v_k}\in V_1 \times ... \times V_k$ and $\widetilde{\N}(v_k)\in F^k$ by setting: $$\widetilde{v_k}=(v_1,...,v_k)$$ such that for each $i$, $v_i=v_{i-1}x_i+k_iy_i$ where $k_i\in F[x_i]$ 
and $$ \widetilde{\N}(v_k)=(N_1,...,N_k) $$ where $N_i=\N_{V_i}(v_i)$.

\begin{prop}\label{B}
Let $A$ and $V_1,...,V_t$ be as above.
Then
\begin{enumerate}
\item $\dim(V_j) = jn+1$.
	\item \label{B1} $V_j$ is an $n$-Kummer space for every $j\geq 0$.
	\item \label{BN} $\N_0(f)=f^n$ and for $j>0$ $$\N_{V_j}(v_{j-1}x_j+k_jy_j)=\N_{V_{j-1}}(v_{j-1})a_j+\N_{F[x_j]/F}(k_j)b_j$$
\end{enumerate}	
\end{prop}

\begin{proof} \ \\
%\begin{enumerate}
(1) This is clear.\\
(2)+(3) For $t=0$ it is clear.

%For $V_1$ compute $\N_{V_1}(fx_1+k_1y_1)=(fx_1+k_1y_1)^n=f^nx_1^n+\N_{F[x_1]/F}(k_1)y_1^n=f^na_1+\N_{F[x_1]/F}(k_1)b_1$. 
For $t>0$ we have for every $v_{t-1} \in V_{t-1}$ and $k_{t}\in F[x_t]$ that
\begin{align*} 
\N_{V_t}(v_{t-1}x_{t}+k_ty_{t}) =(v_{t-1}x_{t}+k_ty_{t})^n = (v_{t-1}x_{t})^n+(k_ty_t)^n\\
= v_{t-1}^nx_{t}^n+\N_{F[x_t]/F}(k_t)y_t^n = \N_{V_{t-1}}(v_{t-1})a_{t}+\N_{F[x_t]/F}(k_t)b_t \in F.
\end{align*}
%\end{enumerate}
\end{proof}	

It turns out that standard $n$-Kummer spaces are connected to presentations of $A$ as a tensor product of symbol algebras in the following way:

\begin{thm} \label{MT1}
Let $A$ and $V_1,...,V_t$ be as above.
\begin{enumerate}	
	\item \label{B2} If $v_t\in V_t$ is such that $\N_{V_t}(v_t)\neq 0$, then one can rewrite $A$ as a product of $t$ symbols with $\N_{V_t}(v_t)$ as one of the slots.
	\item \label{B3} If $v_t\in V_t$ is such that $\widetilde{\N}(v_t)\in {\mul{F}}^k$ then $A$ can be rewritten as $\otimes_{i=1} ^{t} (N_i,*)_n$.	
	\item  \label{MT13} Assume $F[x_i|x_i^n=a_i]$ is a field for each $i$. If  $\N_{V_t}(v_t) = 0$ for some nonzero $v_t\in V_t$, then $A$ can be rewritten as a product of $t-1$ symbols.
\end{enumerate}
\end{thm}

\begin{proof}  \ \ \ \ \ \ \ \ \ \ \ \ \ \ \ \ \ \ \ \ \ \ \ \ \ \ \ \ \ \ \ \ \ \ \ \ \ \ \ \ \ \ \ \ \ \ \ \ \ \ \ \ \ \ \ \ \ \ \ \ \ \ \ \ \ \ \ \ \ \ \ \ \ \ \ \ \
\begin{enumerate}	

	\item 	For $t=1$ write $v=cx_1+ky_1$ for $c \in F$ and $k\in F[x_1]$. The elements $u = x_1^{-1}ky_1$ and $v$ satisfy $uv=\rho vu$, so $$A \isom (v^n,u^n)_n = (\N_{V_1}(v),u^n)_n.$$
	
	For $t>1$, let $A' = (a_1,b_1)_n \otimes \cdots \otimes (a_{t-1},b_{t-1})$ so that \\ $A = A' \otimes (a_t,b_t)$. Let $v_t\in V_{t}$ be such that $\N_{V_t}(v_t) \neq 0$. 
	Write $v_t=v_{t-1}x_{t}+k_ty_{t}$ for $v_{t-1} \in V_{t-1}$ and $k_t\in F[x_{t}]$.
	We know that $\N_{V_t}(v_t)=v_t^n=\N_{V_{t-1}}(v_{t-1})a_{t}+\N_{F[x_t]/F}(k_t)b_{t}\neq 0$. \\There are now two cases:\\
	If $\N_{V_{t-1}}(v_{t-1})=0$ then $\N_{V_t}(v_t)=\N_{F[x_t]/F}(k)b_{t}$, so 
	\begin{align*}
		A &= A' \otimes (a_{t},b_{t})_n \isom A'\otimes (a_{t},\N_{F[x_t]/F}(k_t)b_{t})_n\\
		&=A' \otimes (\N_{V_t}(v_t),a_{t}^{-1})_n,
	\end{align*}
	as claimed. \\
 If $\N_{V_{t-1}}(v_{t-1}) \neq 0$ then by induction we can write $A'$ as a product of $t-1$ symbols of which the last is $(\N_{V_{t-1}}(v_{t-1}),d_{t-1})$ for some $d_{t-1} \in \mul{F}$. 
Now apply proposition \ref{A}(\ref{A2}) to the algebra $$(\N_{V_{t-1}}(v_{t-1}),d_{t-1})_n \otimes (a_{t},b_{t})_n$$ writing it as a tensor product of two symbols where one of the slots is $\N_{V_{t-1}}(v_{t-1})a_{t}+\N_{F[x_t]/F}(k) b_{t} = \N_{V_t}(v_t)$ and we are done.

\item This is analogous to (\ref{B2}), noting that in the last step where we apply  proposition \ref{A}(\ref{A2}) to the algebra $$(\N(v_{t-1}),d_{t-1})_n \otimes (a_{t},b_{t})_n$$ we get that it is isomorphic to $(\N(v_{t-1}),*)_n \otimes (\N(v_t),*)_n$.
\item  Let $t'$ be minimal with respect to the property that $\N$ has a nontrivial zero on $V_{t'}$. Reducing the length of the product of the first $t'$ symbols, we may assume that $t'=t$. Let $A'$ be the product of the first $t-1$ symbols, as before. Let $0\neq v\in V_t$ with $\N(v)=0$. Write $v=v_{t-1}x_t+ky_t$ where $v_{t-1}\in V_{t-1}$ and $k\in F[x_t]$. If $v_{t-1}=0$ then we would have $v_t = k y_t$. Then $0 = \N(v_t)=\N_{F[x_t]/F}(k)b_t$ forces $k = 0$ since $F[x_t]$ is a field, and thus $v = 0$, contrary to assumption. So we may assume $v_{t-1}\neq 0$, and then $\N(v_{t-1})\neq 0$, by minimality, so by part $(1)$ of this proposition we may write this algebra as a product of $t-1$ symbols where the final symbol is $(\N(v_{t-1}),d_{t-1})$ for some $d_{t-1} \in \mul{F}$.\\
 \ \\ 
By \ref{A}(\ref{A1}), \\ 
$(\N(v_{t-1}),d_{t-1})_n\otimes(a_t,b_t)_n \isom (\N(v_{t-1}),*)_n \otimes (\N(v_{t-1}) a_t, \N_{F[x_t]/F}(k)b_t)_n$. \\
But $(\N(v_{t-1}) a_t, \N_{F[x_t]/F}(k)b_t)_n$ splits since $$\N(v_{t-1}) a_t + \N_{F[x_t]/F}(k)b_t = \N(v)=0$$ and $(c,-c)_n$ is split for every $c \in \mul{F}$.

\end{enumerate}

\end{proof}
%
%\begin{rem}
%We can be more specific in \ref{B}(\ref{B2}), 

%\begin{thm} \label{MT1}
%Let $A=\otimes_{i=1}^t (a_i,b_i)_n$ with the standard pairs of generators $x_i,y_i$ for the respective symbols. Assume $F[x_i|x_i^n=a_i]$ is a field for each $i$.
%Assume that % $\{ v\in V_{t-1}|\N(v)=0\}=\{0\}$ but the exist 
%$\N(v) = 0$ for some nonzero $v\in V_t$.
%Then $A$ can be rewritten as a product of $t-1$ symbols.
%\end{thm}
%
%\begin{proof}
%Let $t'$ be minimal with respect to the property that $\N$ has a nontrivial zero on $V_{t'}$. Reducing the length of the product of the first $t'$ symbols, we may assume that $t'=t$. Let $A'$ be the product of the first $t-1$ symbols, as before. Let $0\neq v\in V_t$ with $\N(v)=0$. Write $v=v_{t-1}x_t+ky_t$ where $v_{t-1}\in V_{t-1}$ and $k\in F[x_t]$. If $v_{t-1}=0$ then we would have $v_t = k y_t$ and $0 = \N(v_t)=\N_{F[x_t]/F}(k)b_t$ forces $k = 0$ since $F[x_t]$ is a field and then $v = 0$, contrary to assumption. So we may assume $v_{t-1}\neq 0$, and then $\N(v_{t-1})\neq 0$, by minimality, so we may apply  \ref{B}(\ref{B2}) to $A'$ and write this algebra as a product of $t-1$ symbols where the final symbol is $(\N(v_{t-1}),d_{t-1})$ for some $d_{t-1} \in \mul{F}$. 
%By \ref{A}(\ref{A1}) 
%$(\N(v_{t-1}),d_{t-1})_n\otimes(a_t,b_t)_n \isom (\N(v_{t-1}),*)_n \otimes (\N(v_{t-1}) a_t, \N_{F[x_t]/F}(k)b_t)_n$.
%But $(\N(v_{t-1}) a_t, \N_{F[x_t]/F}(k)b_t)_n$ splits since $\N(v_{t-1}) a_t + \N_{F[x_t]/F}(k)b_t = \N(v)=0$, and $(c,-c)$ splits for every $c \in F$.
%\end{proof}

\begin{thm}\label{MT2}
Let $F$ be a field containing all $n$-th roots of unity, with the property that every homogeneous equation of degree $n$ in $f(n)$ variables has a non-trivial solution.
Then every $A\in \Br(F)$ of exponent $n$ is similar to the product of at most $s=\left\lceil \frac{f(n)}{n} \right\rceil-1$ symbols of degree $n$.
\end{thm}

\begin{proof}
We will show that every product of $s+1$ symbols of degree $n$ is similar to the product of $s$ symbols of degree~$n$ and the theorem will follow after applying the Merkurjev-Suslin theorem.
Let $B=\prod_{i=1}^{s+1} (a_i,b_i)_n$. First we use \ref{P}, so we assume $F[x_i]$ is a field for all $i$.
By \ref{B}(\ref{B1}) we have $V_{s+1}\subset B$, which is a linear space of dimension $(s+1)n +1$ and the norm form on it is homogeneous of degree $n$. But $(s+1)n +1\geq f(n)+1>f(n)$ thus there exist a non zero $v\in V_{s+1}$ such that $\N(v)=0$.
Thus applying \ref{MT1} we get that $B$ is similar to the product of at most $s=\left\lceil \frac{f(n)}{n} \right\rceil-1$ symbols.

\end{proof}

\begin{thm}\label{MTP}
Let $F$ be a $C_m$ field containing all $n$-th roots of unity and $A\in\Br_n(F)$ be of exponent $n$. Then, 

\label{C1} $$\len(n,A)\leq n^{m-1}-1$$ that is $\len(n,F)\leq n^{m-1}-1$.
	%\item \label{C2} $\ind(A)\leq n^{n^{m-1}-1}$, in particular if $A$ is an algebra of exponent $n$ and degree $d>n^{n^{m-1}-1}$ it is not a division algebra! 

\end{thm}

\begin{proof}
%Clearly (\ref{C2}) follows from (\ref{C1}).
%For %(\ref{C1}) 
Any $C_m$ field satisfies the property that every homogeneous equation of degree $n$ in more then $n^m$ variables has a solution.
Thus by \ref{MT2} every $A$ of exponent $n$ is similar to the product of at most $\len(n,F)\leq \frac{n^m}{n}-1=n^{m-1}-1$ symbols of degree $n$, proving the theorem.
\end{proof}

%\begin{rem}
%Notice that one can reproduce the proof of theorem \ref{} in the language of $K_2(F)$ to get an equivalent statement for the length of elements in $K_2(F)/n$, indeed one can reproduce the forms $\N_i$ using an iterative application of the relations \ref{} which hold in $K_2(F)/n$.
%%want to show a product of $t$ symbols is equivalent to the product of $t-1$ symbols we can either try to find a representation of the $t$ symbols where the last one has an $n$-Kummer element with norm $0$, which was the strategy we used above, which basically tells us the the "field" generated by that zero norm $n$-central element is NOT a field.
%%Using this idea we can improve the Norm form at the last step to have one more variable, that is to consider the following Norm form for the linear space: $F+V_t$ defining $\N(f+v_t)=\N_{F[v_t]/F}(f+v_t)$ which adds one more variable, however this does not improve the above results.
%\end{rem}

%\newpage

\section{Improving the result for non-prime exponent}
%Let $F$ be a $C_m$ field (or any field with a property as in \ref{MT2}). 
It is a standard fact that we have a primary decomposition for the Brauer group, that is, if $\exp(A)=n=\prod_{i=1} ^t p_i^{n_i}$ where $p_1,...,p_t$ are different primes, then $A=\prod_{i=1} ^t A_i$ where $A_i$ is of exponent $p_i ^{n_i}$.
The first improvement comes from writing each of the $A_i$ as a product of symbols.
% and adding them all together to write A as a product of symbols.
\begin{prop}
Assume $(n_1,n_2)=1$ and set $\rho_{n_1}=\rho_{n_1n_2}^{n_2}$ and $\rho_{n_2}~=~\rho_{n_1n_2}^{n_1}$. Then $(a_1,b_1)_{n_1}\otimes(a_2,b_2)_{n_2}\cong (a_1^{n_2}a_2^{n_1},b_1^{n_2k}b_2^{n_1s})_{n_1n_2}$ where $$sn_1+kn_2=1 \ \hbox{mod} \ n_1n_2.$$

\end{prop}
\begin{proof}
%Write $(a_i,b_i)_{n_i}=F[x_i,y_i|x_i^{n_i}=a_i, y_i^{n_i}=b_i, y_ix_i=\rho_{n_i}x_iy_i]$ 
Let $x_i,y_i$ be the standard generators for $(a_i,b_i)_{n_i}$, $i=1,2$.
Consider the elements $u=x_1x_2, v=y_1^{k}y_2^{s}$ and compute:
$u^{n_1n_2}=a_1^{n_2}a_2^{n_1}$, $v^{n_1n_2}=b_1^{n_2k}b_2^{n_1s}$ and $vu=y_1^ky_2^sx_1x_2=\rho_{n_1}^kx_1y_1^ky_2^sx_2=\rho_{n_1}^k\rho_{n_2}^sx_1x_2y_1^ky_2^s=\rho_{n_1n_2}^{kn_2+sn_1}uv=\rho_{n_1n_2}uv$.
Thus the proposition is proved.
\end{proof}

\begin{cor}
If $A$ is as above, then $\len(n,A) \leq \max\{\len(p_i ^{n_i},A_i)\}$.
\end{cor}

Thus it is better to consider the case where $n=p^t$ for a prime~$p$.

Next we are going to use the well known divisibility of symbol algebras to further improve the our result on symbol length.

\begin{prop} "Divisibility of Symbols" \cite[page 537]{RB} \\
If $A=(\alpha,\beta)_s$ then $A\sim (\alpha,\beta)_{sk} ^k$, assuming $\rho_{sk}\in F$.
\end{prop}

\begin{thm}\label{FMT}
Let $F$ be a $C_m$ field containing all $n$-th roots of unity. If $A\in \Br_n(F)$ is of exponent $n=p^t$, then: \ \\
$A=\otimes_{i=1}^t C_i$ where $C_i=\otimes_{j=1}^{p^{m-1}-1} (\alpha_{i,j},\beta_{i,j})_{p^{i}}$. In particular $$\len(p^t,F)\leq t(p^{m-1}-1).$$
%\begin{enumerate}
	%\item \label{IM1} $A=\prod_{i=1}^t C_i$ where $C_i=\prod_{j=1}^{\len(m,p)} (\alpha_{i,j},\beta_{i,j})_{p^{i}}$, in particular $\len(p^t,F)\leq t\len(p,F)=t(p^{m-1}-1)$.
%	\item \label{IM2} $\ind(A)\leq p^{\frac{t(t+1)}{2}\len(m,p)} <  {p^{t}}^{{p^t}^{m-1}-1}$.
%\end{enumerate}
\end{thm}

\begin{proof}
For $t=1$ this is theorem \ref{MTP} with $n=p$.
%For $t=2$: Let $A$ be as above.
%Define $B=A^p$ and notice that $B$ is of exponent $p$ thus by \ref{MT1} $B\sim \prod_{j=1}^{\len(m,p)} (\alpha_{j},\beta_{j})_{p}$.
%Define $C_2=\prod_{j=1}^{\len(m,p)} (\alpha_{2,j},\beta_{2,j})_{p^2}$ where $\alpha_{2,j}=\alpha_j$ and $\beta_{2,j}=\beta_j$, and notice that $C_2^p\sim B$.
%Now consider $C_1=A\otimes C_2^{-1}$, then $C_1^p=A^p\otimes C_2^{-p}\sim B\otimes B^{-1}\sim 1$ thus $C_1$ is of exponent $p$, and again by \ref{MT1} $C_1\sim \prod_{j=1}^{\len(m,p)} (\alpha_{1,j},\beta_{1,j})_{p}$ implying 
%$$A\sim C_1\otimes C_2$$ were $C_1$ and $C_2$ are as in the theorem.

For $t=s+1$: Let $A$ be as above.
Define $B=A^p$ and notice that $B$ is of exponent $p^s$. Thus by induction $B\sim \otimes_{i=1}^{s} C'_i$ where\\ $C'_i=\otimes_{j=1} ^{p^{m-1}-1} (a_{i,j},b_{i,j})_{p^{i}}$.
For $i=2,...,s+1$ define $$C_i=\otimes_{j=1}^{p^{m-1}-1} (\alpha_{i,j},\beta_{i,j})_{p^{i}}$$ where $ \alpha_{i,j}=a_{i-1,j}; \beta_{i,j}=b_{i-1,j}$
and define $B'=\otimes_{i=2}^{s+1} C_i$. \\ Notice $B'^p=B$. Thus considering $C_1=A\otimes B'^{-1}$, we get \\ $C_1^p=A^p\otimes B'^{-p}\sim B\otimes B^{-1}\sim 1$. Thus $C_1$ is of exponent $p$, and by \ref{MT1} $C_1\sim \otimes_{j=1}^{p^{m-1}-1} (\alpha_{1,j},\beta_{1,j})_{p},$ implying 
$$A\sim C_1\otimes B'=\otimes_{i=1} ^{s+1} C_i$$ where $C_1,... ,C_{s+1}$ are as in the theorem.
\end{proof}

%\newpage

\section{Fields containing a $C_m$ field}

In this section we consider the more general case where the base field~$F$ contains a $C_m$ field.
Notice that this class of fields includes fields such as $\mathbb{C}(x_i,y_i,i\in \mathbb{Z})$, where there in no hope of finding bounds which are a function of the exponent alone as was explained in the background section.

Thus consider a central simple algebra $A$, of exponent $p^t$ and degree $p^s$ over a field $F$ containing a $C_m$ field $L$.

The main idea is that even though $F$ might not be finitely generated over $L$, $A$ is defined over a ``smaller" field $K$, which is finitely generated over $L$. This is expressed by considering the essential dimension of $A$ over $L$, denoted $\ed_L(A)$.

\begin{defn}
Let $A$ be as above. The essential dimension $\ed_L(A)$ of $A$ over $L$ is, $$\min \{ \trdeg_F(K) | L \subseteq K \subseteq F \mbox{ \ and \ } A=A'\otimes_K F \mbox{ \ for \ } A'\in \Br(K)\}$$
%where $ \res$ is the restriction map from $L$ to $K$.
\end{defn}

It is known that $\ed_L(A)$ is finite and bounded by functions of the degree of $A$. For known bounds and more on the subject we refer the reader to \cite{LRRS} and \cite{M}.
Consider $A$ as above. Then there exists a field $L \subseteq K\subseteq F$ with $\trdeg_F(K)=\ed_L(A)$ and an $K$-csa $A'$ such that $A'\otimes_K F=A$. Thus if we write $A'$ as a sum of symbols (which we can since $K$ is a $C_{m+\ed_L(A)}$ field) we can tensor up to $F$ and write $A$ as a product of symbols.

The only problem is that the exponent of $A'$ can be bigger than that of $A$, and the resulting presentation of $A$ will use symbols of higher degrees than needed.

In order to deal with that we need to consider a specialized essential dimension, that is 
\begin{displaymath}\ed_{exp}(A)= \min \{ \trdeg_L(K) | L \subseteq K \subseteq F ; A=A'\otimes_K F; A'\in \Br(K); \exp(A')=\exp(A)\}.
\end{displaymath}

\begin{prop}\label{ed}
$\ed_{exp}(A)\leq \ed_F(A)+\dim(\SB(D))$ where $D$ is the underlying division algebra of $A'^{\otimes \exp(A)}$, $A'$ is as in the definition of the essential dimension and $\SB(D)$ is the Severi-Brauer variety of $D$.
\end{prop}

\begin{proof}
It is enough to find a field $K\subseteq M \subseteq F$ such that $\exp(A'_M)~=~\exp(A)$ and $\trdeg_L(M)\leq \ed_L(A)+\dim(\SB(D))$.
Notice that by the definition of $A'$, $F$ is a field satisfying $\exp(A'_F)=\exp(A)$. In particular $D_F \sim {A'_F}^{\otimes \exp(A)} \sim F$, so there is a rational point on $\SB(D)_F$. In other words there is a specialization of the function field of $\SB(D)$,
$\nu:K(\SB(D))\rightarrow F$. Let $K\subseteq M\subseteq F$ be the image of $\nu$. Clearly $M$ satisfies all our requirements and $\trdeg_L(M)\leq  \ed_L(A)+\dim(\SB(D))$.

\end{proof}

\begin{thm}\label{FT}
Let $A$ be as above. Then the symbol length of $A$ is bounded by $\len(p^t,M)$ where $M$ is a $C_{m+\ed_L(A)+p^{s-t}-1}$ field.
\end{thm}

\begin{proof}
Let $M$ be as in the proof of proposition \ref{ed}. We want to bound the transcendance degree of $M$.
Notice that $\dim(\SB(D))=\ind(D)-1$ thus we want to bound $\ind(D)$. Since $D\sim A'^{\exp(A)}$ we have $\ind(D)~\leq~\frac{\ind(A')}{\exp(A')}$. Also, $\ind(A')\leq \deg(A)$ as $A'\otimes F=A$ and by assumption $\exp(A')~=~\exp(A)$. Thus, $\ind(D)\leq  \frac{\ind(A')}{\exp(A')} \leq \frac{\deg(A)}{\exp(A)}=p^{s-t}$. It follows that $M$ is a $C_{m+\ed_L(A)+p^{s-t}-1}$ field and applying theorem \ref{FMT} we see that $\len(p^s,p^t,A)~\leq~\len(p^t,M)$ where $M$ is as in the theorem.
\end{proof}

\begin{rem}
If we replace $A$ by its underlying division algebra, $D_A$, in the above we get the bound $\len(m+\ed_L(D_A)+\frac{\ind(A)}{\exp(A)}-1,p^t) $, which might seem smaller as $\ind(A)\leq \deg(A)$. However it might also happen that $\ed_L(D_A)>\ed_L(A)$. For example, if $D$ is a generic division algebra of index $4$ over a field containing $\rho_4$ we know by \cite{M} that $\ed(D)=5$ where as by \cite{LRRS} $\ed(\M_2(D))\leq 4$ since $D$ is similar to the tensor product of two symbols of degree $4$ and $2$ respectively.

%Notice the bound is a function of $\ind(A),\exp(A)$ and $\ed_F(A)$. Now index and exponent are invariants of the class of $A$ but the essential dimension is not. It might happen that $\ed_F(A)>\ed(\M_k(A))$ in which case we will choose to represent $\M_k(A)$ as a product of symbols to get a better bound. Thus in theorem \ref{FT} $\ed_F(A)$ can be replaced with $\ed([A])$ where 
%$\ed([A])=\min\{\ed(B)|[B]=[A] \}$.
\end{rem}
%
%\newpage

\section{The Brauer dimension of a $C_m$ field}
For $F$ a $C_m$ field and $A\in \Br(F)$ of exponent $p^n$, we will show that $\ind(A)\leq \exp(A)^{p^{m-1}-1}$, that is $\Brdim_p(F)\leq p^{m-1}-1$.
We first reduce to the exponent $p$ case and deduce the theorem from our previous results on symbol length.

\begin{prop}\label{IND}
Suppose $F$ and all its algebraic extensions, $L$, have the property that for all central
simple $A/L$ of exponent $p$ satisfies, $\ind(A)\leq p^s$. Then, any $A/F$ of exponent $p^n$ satisfies,
$\ind(A)\leq p^{ns}$.
\end{prop}

\begin{proof}
Clearly the case $n=1$ holds by assumption.
Let $A$ be of exponent $p^{n+1}$.
Consider $B=A^p$. Then $B$ is of exponent $p^n$ and we have $\ind(B)\leq \exp(B)^s$.
Let $L$ be a splitting field for $B$ with $[L:F]~=~\ind(B)$.
Also consider $A_L\in Br(L)$. We have $(A_L)^p=A^p_L=B_L=1$, thus $A_L$ is of exponent $p$.
Now by our assamption on $F$ and its algebraic extensions we have that $\ind(A_L)\leq p^s$.
Take a splitting field $L\subset K$ of $A_L$ with $[K:L]=\ind(A_L)$ and consider $K$ as an extension of $F$. Then 
$$[K:F]=[K:L][L:F]=\ind(B)\ind(A_L)\leq p^{ns}p^s=p^{(n+1)s} \ \hbox{and} \ A_K=1.$$ Thus $\ind(A)\leq  p^{(n+1)s}$ as needed.
\end{proof}

\begin{prop}\label{ind1}
Let $F$ be a $C_m$ field and $L$ be any algebraic extension of $F$. For every $A/L$ of exponent~$p$ we have $ind(A)\leq p^s$.
\end{prop}

\begin{proof}
Since $L$ is algebraic over $F$ and $F$ is $C_m$ so is $L$. Now the proposition follows from Theorem \ref{MTP}.
\end{proof}

Combining \ref{IND} and \ref{ind1} we get:

\begin{thm}\label{Bd}
 If $F$ is a $C_m$ field, then $\Brdim_p(F)\leq p^{m-1}-1$.
\end{thm}

%\newpage
\section{The case of $p$-Algebras}
In this section we deal with the case where $F$ has characteristic $p$ and $A\in \Br_{p^e}(F)$.
It turns out that things are much simpler in this case, both for general fields as shown by Florence in \cite{MF} and for $C_m$ fields where things basically follow from an exercise in Serre, \cite{Serre}.
\subsection{General fields of characteristic $p$} \ \\
In \cite{MF} Florence proves the following theorem:
\begin{thm}\label{fl} 
If $F$ is a field with $\ch(F)=p$ and $A$ is an $F$-csa of index $p^n$ and exponent $p^e$, then $A$ is similar to the product of at most $p^n-1$ cyclic algebras.
\end{thm}
Here is a short proof along the lines of \cite{MF}.
The idea is based on the following two well known theorems.
\begin{enumerate}  
\item Let $A$ be a $p$-algebra of exponent $p^n$, then $$F^{\frac{1}{p^n}}=F[x_f, f\in \mul{F}| x_f^{p^n}=f]$$ splits $A$ (\cite{RB} page 575, exercises 30,31).
\item \label{Albp} If $A\in \Br(F)$ is split by a purely inseparable extension of the form $K=F[x_1,...,x_t|x_i^{n_i}=\alpha_i\in\mul{F}],$ then $A$ is similar to the tensor product $A=\otimes_{i=1}^t A_i,$ where $A_i$ is a cyclic $p$-algebra with maximal subfield $K_i=F[x|x^{n_i}=\alpha_i]$, and in particular the symbol length of $A$ is at most $t$ (Albert, \cite{Al} theorem 28, page 108).
\end{enumerate}

The idea is then to start with the splitting field $F^{\frac{1}{p^e}}$ which (in general) is infinite dimensional over $F$, and to find a finite dimensional subfield splitting $A$. Then one uses Albert's theorem to present $A$ as the product of cyclic $p$-algebras.

\begin{proof} of Theorem \ref{fl}. \\
Let $\SB(A)$ denote the Severi-Brauer variety of $A$ and $F\SB(A))$ its function field. Since $F^{\frac{1}{p^e}}$ splits $A$ there is a place
$\nu: F(\SB(A))\rightarrow F^{\frac{1}{p^e}}$. Let $K\subset F^{\frac{1}{p^e}}$ be the image of $\nu$.
First notice that $[K:F]<\infty$ since $\SB(A)$ is finitely generated and $F^{\frac{1}{p^e}}$ is algebraic over $F$.
It remains to bound $[K:F]$. Since dimension is invariant under scalar extensions, it is enough to bound $[K\otimes F^{sep}: F^{sep}]$, but $A\otimes F^{sep}$ is split and thus $\SB(A)\times F^{sep}\cong \mathbb{P}^{(p^n-1)}(F^{sep})$, which implies 
$$F^{sep}[\SB(A)]\cong F^{sep}[x_1,...,x_{p^n-1}].$$ Thus the image of $\nu$ is generated by $p^n-1$ elements. Now the image of every element is algebraic of degree at most $p^e$, which implies \linebreak $[K\otimes F^{sep}:F\otimes F^{sep}]\leq p^{e(p^n-1)}$ and we are done by Albert's theorem above.
\end{proof}

\subsection{$C_m$ fields}
Now we assume $F$ is a $C_m$ field with $\ch(F)=p$.
As we just saw above we want to bound the dimension of the image of $\nu$.
But since $F$ is $C_m$, an exercise in Serre (\cite{Serre} page 89 exercise  3) tells us:
\begin{prop}
If $F$ is as above, then $[F^{\frac{1}{p^e}}:F]\leq p^{em}$.
\end{prop}
\begin{proof}
Assume we found a subfield $K\subseteq F^{\frac{1}{p^e}}$ of dimension $p^{em}$ over $F$. We now show $K= F^{\frac{1}{p^e}}$.
Pick a basis $\{ k_1,...k_{p^{em}}\}$ for $K$ over $F$ such that $k_i^{p^e}=\alpha_i$. Let $y\in F^{\frac{1}{p^e}}$ so that $y^{p^e}=\beta$. We will show $y\in K$.
Consider the homogeneous equation $$\sum_{i=1}^{p^{em}}x_i^{p^e}\alpha_i=x_{p^{em}+1}^{p^e}\beta.$$ Since this is of degree $p^e$ in $p^{em}+1>p^{em}$ variables, the $C_m$ property implies there is a non-trivial solution $\overline{s}=(x_1,...,x_{p^{em}},x_{p^{em}+1})$. It is enough to show $ x_{p^{em}+1}$ is not zero, as in this case we see the above element $t=\frac{1}{x_{p^{em}+1}}\sum_{i=1}^{p^{em}}x_ik_i\in K$ satisfies $t^{p^e}=\beta$, and thus the element $y\in F^{\frac{1}{p^e}}$ is actually in $K$.
To see $ x_{p^{em}+1}\neq 0$ assume it is zero. Then the element $u=\sum_{i=1}^{p^{em}}x_ik_i\in K$  satisfies $u^{p^e}=0$ implying $u=0$ so $\overline{s}=\overline{0}$ contrary to the assumption $\overline{s} \neq \overline{0}$.
\end{proof}
\begin{cor}
Let $F$ be as above and $A\in \Br(F)$ be of exponent $p^e$. Then $A$ is similar to the product of at most $m$ cyclic $p$ algebras of degree~$p^e$, and in particular $\len(p^e,F)\leq m$ and $\Brdim(F)\leq m$.
\end{cor}

%\newpage
 
\section{Symbol length in $K^M_2(F)/nK^M_2(F)$}
In this section we observe that the basic relations we used in section~$3$ also hold for $K^M_2(F)/nK^M_2(F)$, and thus the main theorem can be formulated in this context. (Notice that roots of unity are not needed.)

We then show explicitly how to shorten the symbol length of an element in $K^M_2(F)/2K^M_2(F)$ over a $C_2$ field, which by the theorem should be one symbol. This computation illustrates the process of shortening the symbol length for $p=2$ but clearly enables one to see the same process will work (but will be much longer) for any prime power $n$.

The following relations are well known over any field $F$.
\begin{prop}
If $\{a,b\},\{c,d\}\in K^M_2(F)/nK^M_2(F)$, then:
\begin{enumerate}
\item $\{a,1-a\}=0$ for $a\neq 1,0$. \label{p}
\item $\{f,1\}=0$ for $f\in\mul{F}$. \label{p0}
\item $\{a,b\}=\{f^na,b\}$ for $f\in \mul{F}.$ \label{p1}
\item $\{a,b\}=\{a,\N_{K/F}(k)b\}$ for $K=F[\sqrt[n]{a}]$ and $k\in \mul{K}.$ \label{p2}
\item $\{f,-f\}=0$ for $f\in \mul{F}.$ \label{p3}
\item If $a+b\in \mul{F}$, then $\{a,b\}=\{a+b,-a^{-1}b\}.$ \label{p4}
\item $\{a,b\}+\{c,d\}=\{a,bd^{-1}\}+\{ac,d\}.$ \label{p5}
\end{enumerate}
\end{prop}

\begin{proof}
\begin{enumerate}
\item This is one of the defining relations of $K^M_2(F)$.
\item This is true even in $K^M_2(F)$. Notice $$\{f,1\}=\{f,1^2\}=\{f,1\}+\{f,1\},$$ and thus $\{f,1\}=0$. 
\item Compute $\{f^na,b\}=\{f^n,b\}+\{a,b\}=n\{f,b\}+\{a,b\}=\{a,b\}.$
\item Let $k\in \mul{K}$. Using the projection formula one computes 
\begin{align*}
\{a,\N_{K/F}(k)\}&=\Cor_{K/F}(\{a,k\})=\Cor_{K/F}(\{(\sqrt[n]{a})^n,k\})\\ 
&=\Cor_{K/F}(n\{\sqrt[n]{a},k\})=0.
\end{align*}
Thus $\{a,\N_{K/F}(k)b\}=\{a,b\}+\{a,\N_{K/F}(k)\}=\{a,b\}.$
\item Consider $K=F[\sqrt[n]{f}]$ and compute $$\N_{K/F}(-\sqrt[n]{f})=(-1)^n\N_{K/F}(\sqrt[n]{f})=(-1)^n(-1)^{n-1}f=-f.$$ Thus by (\ref{p0}) and (\ref{p2}) we have \\$0=\{f,1\}=\{f,\N_{K/F}(-\sqrt[n]{f})1\}=\{f,-f\}$.
\item Compute 
\begin{align*}
\{a+b,-a^{-1}b\}&=\{a(1+a^{-1}b),-a^{-1}b\}\\
&=\{a,-a^{-1}b\}+\{1+a^{-1}b,-a^{-1}b\}\stackrel{(\ref{p})}{=}\{a,-a^{-1}b\}\\
&=\{a,-a^{-1}\}+\{a,b\}=-\{a,-a\}+\{a,b\}\stackrel{(\ref{p3})}{=}\{a,b\}.
\end{align*}
\item Compute 
\begin{align*}
\{a,bd^{-1}\}+\{ac,d\}&=\{a,b\}+\{a,d^{-1}\}+\{a,d\}+\{c,d\}=\{a,b\}+\{c,d\}.
\end{align*}
\end{enumerate}
\end{proof}

\begin{thm}
Let $F$ be a $C_m$ field and let $\alpha \in K^M_2(F)/nK^M_2(F)$, where $n=p^k$. Then $\alpha$ can be written as the sum of at most $k(p^{m-1}-1)$ symbols.
\end{thm}

We are ready for the example.
Let $F$ be a $C_2$ field and let $$\alpha=\{a_1,b_1\}+\{a_2,b_2\}\in K^M_2(F)/2K^M_2(F).$$
We will show that $\alpha$ can be rewritten as a single symbol as stated in the theorem.
Recall the norm forms from section $2$ attached to $\alpha$, namely:
Letting $x_i,y_i$ be the standard generators for the two quaternions, $(a_i,b_i)_2$, $L_i=F[x_i]=F[\sqrt{a_i}]$, we define
$$V_1=Fx_1+L_1y_1$$ 
$$V_2=V_1x_2+L_2y_2$$ and their norm forms
$$N_1(v_1=fx_1+l_1y_1)=f^2a_1+\N_{L_1/F}(l_1)b_1$$
$$N_2(v_2=v_1x_2+l_2y_2)=N_1(v_1)a_2+\N_{L_2/F}(l_2)b_2$$
We may assume $L_i's$ are fields.
Notice that $\deg(N_2)=2$ and $\dim(N_2)=5>2^2$, and thus there exist a non-zero $v_2=v_1x_2+l_2y_2\in V_2$ where $v_1=fx_1+l_1y_1$, such that $\N_2(v_2)=0$. 
If $v_1=0,$ we get $\N_{L_2/F}(l_2)b_2=0$, implying $v_2=0$ and thus $v_1\neq 0$.
Also, if $\N_1(v_1)=0$ we have $$\{a_1,b_1\}\stackrel{(\ref{p2})}{=}\{a_1,\N_{L_1/F}(l_1)b_1\} \stackrel{(\ref{p1})}{=}\{f^2a_1,\N{L_1/F}(l_1)b_1\} \stackrel{\N_1(v_1)=0}{=}(c,-c)\stackrel{(\ref{p3})}{=}0$$ and $\alpha$ is one symbol.
Thus we assume $\N_1(v_1)\neq 0$.
From the above we write \\
 \ \\
$\{a_1,b_1\}+\{a_2,b_2\}=\{f^2a_1,\N_{L_1/F}(l_1)b_1\}+\{a_2,b_2\}\\
\stackrel{(\ref{p4},\ref{p2})}{=} \{\N_1(v_1),(f^2a_1)^{-1}\N_{L_1/F}(l_1)b_1\}+\{a_2,\N_{L_2/F}(l_2)b_2\}\\ 
\stackrel{(\ref{p5})}{=} \{\N_1(v_1),((f^2a_1)^{-1}\N_{L_1/F}(l_1)b_1)(\N_{L_2/F}(l_2)b_2)^{-1}\}+\{\N_1(v_1)a_2,\N_{L_2/F}(l_2)b_2\} \\
 \stackrel{\N_2(v_2)=0}{=}\{\N_1(v_1),(f^2a_1)^{-1}\N_{L_1/F}(l_1)b_1(\N_{L_2/F}(l_2)b_2)^{-1}\}+\{c,-c\}\\
\stackrel{(\ref{p3})}{=} \{\N_1(v_1),(f^2a_1)^{-1}\N_{L_1/F}(l_1)b_1(\N_{L_2/F}(l_2)b_2)^{-1}\}.$


\newpage
\begin{thebibliography}{99}



\bibitem{Al} A. A. Albert, Structure of algebras, AMS Coll. Pub., vol. 24, AMS, Providence, RI, 1961, revised printing.
%%
\bibitem{A0} A. A. Albert and H. Hasse. A determination of all normal division algebras over an algebraic number field. Trans. Amer. Math. Soc., 34 (1932), 722–726.




\bibitem{Art82} M. Artin, Brauer-Severi varieties, Brauer groups in ring theory and algebraic geometry
(Wilrijk, 1981), Lecture Notes in Math., vol. 917, Springer, Berlin, 1982, pp. 194–210.

\bibitem{Ketura} A. Auel, E. Brussel, S. Garibaldi and U. Vishne, Open problems on central simple algebras, Transformation Groups, March 2011, Volume 16, Issue 1, pp 219-264.  

\bibitem{BMT} E. Brussel, K. Mckinnie and E. Tengan, Cyclic Length in the Tame Brauer Group of the Function Field of a $p$-Adic Curve, preprint.

\bibitem{ACT} A. Chapman, Polynomial equations over division rings, master thesis, Bar Ilan, 2009.


\bibitem{GS} P. Gille and T. Szamuely, Central simple algebras and Galois cohomology, Cambridge studies in advanced mathematics 101 (2006).




\bibitem{BHN} R. Brauer, H. Hasse and E. Noether. Beweis eines Hauptaatzes in der Theorie der Algebren. J. Math., 167 (1931), 399–404.


%%

\bibitem{MF} M. Florence, Central simple algebras of index $p^n$ in characteristic $p$, preprint http://www.math.jussieu.fr/$\sim$florence/p-alg-compositio-v2.pdf.



\bibitem{HHK09} D. Harbater, J. Hartmann, and D. Krashen, Applications of patching to quadratic
forms and central simple algebras, Invent. Math. 178 (2009), 231–263.

\bibitem{Jac} N. Jacobson, Finite Dimensional Division Algebras Over Fields, Springer-Verlag, 1996.

\bibitem{dJ04} A.J. de Jong, The period-index problem for the Brauer group of an algebraic surface,
Duke Math. J. 123 (2004), no. 1, 71–94.







\bibitem{Lang} S. Lang, On quasi-algebraic closure, Ann. of Math. 55, 373-390 (1952).



\bibitem{Lie08} M. Lieblich, Twisted sheaves and the period-index problem, Compositio Math. 144
(2008), 1–31.


\bibitem{LRRS} M. Lorenz, L.H. Rowen, Z. Reichstein, and D.J. Saltman, The field of definition of a division algebra, J. London Math. Soc. 68 no. 3:651-679, 2003.

\bibitem{MAT} E. Matzri, $\mathbb{Z}_3\times \mathbb{Z}_3$ crossed products, to appear in J. of Algebra. 

\bibitem{M} A.S Merkurjev, Essential $p$-dimension of $ \operatorname{\mathbf{PGL}}(p^2)$. J. Amer. Math. Soc. 23:693-712, 2010.



\bibitem{MS} A.S. Merkurjev and A.A. Suslin, K-cohomology of Severi-Brauer varieties and the
norm residue homomorphism, Math. USSR Izv. 21 (1983), no. 2, 307–340.
%%
%%\bibitem{AS} S.A. Amitsur and D. Saltman, Generic abelian crossed products and p-algebras. J. Algebra, 51(1):76-87, 1978.
%%
%%
%%
%%
%%

%\bibitem{RS} L.H. Rowen and D. Saltman, Prime to p extensions of division algebras, Israel J. Math., 78(2-3):197-207, 1992.




\bibitem{Nagata} M. Nagata, Note on a paper Lang concerning quasi-algebraic closure, Mem. Univ. Kyoto 30, 237-241 (1957).

\bibitem{RT} S. Rosset and J. Tate, A reciprocity low for $K_2$-traces, Comment. Math. Helv. 58(1983), 38-47.
	


\bibitem{RB} L.H. Rowen, Graduate algebra: non commutative view. Graduate Studies in Mathematics, 91.
American Mathematical Society, Providence, RI, 2008. xxi+648 pp.


\bibitem{RSD} L.H. Rowen and D.J. Saltman, Dihedral algebras are cyclic, Proceedings of the American
Mathematical Society, Vol. 84, No. 2 (Feb. 1982), 162-164.

\bibitem{Sal97a} D.J. Saltman , Division algebras over p-adic curves, J. Ramanujan Math. Soc. 12 (1997), 25–47.

\bibitem{SLN} D. J. Saltman, Lectures on Division Algebras, CBMS Number 94, Conference on
Division Algebras held at Colorado State University, Fort Collins, June 14-18,
1998, American Mathematical Society, Rhode Island.


\bibitem{Serre} J.P. Serre, Galois cohomology, translated from french by P.Ion, Springer (1997).

%\bibitem{SV} A.S. Sivatski $(Z/2Z)^4$-Crossed products of exponent 2, preprint.
%%

\bibitem{T1} J.P. Tignol, Cyclic algebras of small exponent, American Mathematical Society. Proceedings, Vol. 89, no. 4, 587-588 (1983).
%%
%%
\bibitem{T2}J.P. Tignol, Corps `a involution neutralisÕ¥s par une extension abÕ¥lienne Õ¥lÕ¥mentaire,
Springer Lecture Notes in Math. 844 (1981), 1-34.

%\bibitem{W} J.H.M Wedderburn, On division algebras, Trans. Am. Math. Soc. 9, 129-135.



\bibitem{Wed21} J.H.M. Wedderburn, On division algebras, Trans. Amer. Math. Soc. 22 (1921), 129–
135.




%%
%%\bibitem{T2}J.P. Tignol, Corps `a involution neutralis´es par une extension ab´elienne ´el´ementaire,
%%Springer Lecture Notes in Math. 844 (1981), 1-34.
%
\end{thebibliography}
\end{document}